\documentclass[11pt,reqno]{amsart}
\usepackage{bbm}
\usepackage{}
\usepackage{amsfonts}
\usepackage{amssymb,bm}
\DeclareMathOperator{\curl}{curl}
\DeclareMathOperator{\CURL}{CURL}
\DeclareMathOperator{\divg}{div}

\hoffset=- 2cm \voffset=- 0cm 
 \theoremstyle{plain}
\newtheorem{Thm}{Theorem}

\newtheorem{Prop}[Thm]{Proposition}
\newtheorem{Rem}[Thm]{Remark}
\newtheorem{Lem}[Thm]{Lemma}
\newtheorem{Cor}[Thm]{Corollary}

\newtheorem{Claim}[Thm]{Claim}

\usepackage{amssymb}
\usepackage[active]{srcltx}
\usepackage{color}

\errorcontextlines=0 \numberwithin{equation}{section}
\numberwithin{Thm}{section}

\begin{document}
\large
%Topmatter

\title[Fractional integral operator and its applications]
{Fractional integral operator for $L^1$ vector fields and its applications}

\author[]{Zhibing Zhang}

\address{Zhibing Zhang: School of Mathematics and Physics, Anhui University of Technology, Ma'anshan 243032, PR China; }

\email{zhibingzhang29@126.com}
\thanks{}



\begin{abstract}
This paper studies fractional integral operator for vector fields in weighted $L^1$. Using the estimates on fractional integral operator and Stein-Weiss inequalities, we can give a new proof for a class of Caffarelli-Kohn-Nirenberg inequalities and establish new $\divg$-$\curl$ inequalities for vector fields.
\end{abstract}
\maketitle
%end topmatter

\section{Introduction}
For $0<\lambda<n$, the fractional integral operator $\mathrm{I}_{\lambda}$ is defined by $$\mathrm{I}_{\lambda}f(x)=(-\Delta)^{-\frac{\lambda}{2}}f(x)=C_{n,-\lambda}\int_{\mathbb{R}^n}\frac{f(y)}{|x-y|^{n-\lambda}}dy,$$
which is also called the Riesz potential. It is well-known that $\mathrm{I}_{\lambda}$ is a bounded linear operator from $L^p(\mathbb{R}^n)$ to $L^q(\mathbb{R}^n)$, where $1/q=1/p-\lambda/n$ and $1<p<n/\lambda$, which is so-called Hardy-Littlewood-Sobolev inequality. Stein and Weiss \cite{SteinW} established a doubly weighted generalization as follows.
\begin{Thm}[Stein-Weiss]\label{SW}
Let $0<\lambda<n$, $1<p<\infty$, $\alpha<\frac{n}{p'}$, $\beta<\frac{n}{q}$, $\alpha+\beta\geq 0$, and $\frac{1}{q}=\frac{1}{p}+\frac{\alpha+\beta-\lambda}{n}$. If $p\leq q<\infty$, then there exists a constant $C$, independent of $f$, such that
$$\left\||x|^{-\beta}\mathrm{I}_{\lambda}f\right\|_{L^q}\leq C\left\||x|^{\alpha}f \right\|_{L^p}.$$
\end{Thm}
De N\'{a}poli, Drelichman and Dur\'{a}n \cite{DDD2011} found that if we consider radially symmetric functions, then Theorem \ref{SW} holds for a wider range of exponents. In fact, they pointed out that for radially symmetric functions if $p>1$, Theorem \ref{SW} holds for $\alpha+\beta\geq (n-1)(\frac{1}{q}-\frac{1}{p})$; if $p=1$, Theorem \ref{SW} holds for $\alpha+\beta>(n-1)(\frac{1}{q}-1)$.

As we are interested in the extreme case $p=1$ of Stein-Weiss inequality, we turn to the end-point estimates for $L^1$ vector fields, which was pioneered by Bourgain and Brezis \cite{BB2002}. For any vector-valued function $\mathbf{f}\in L^1(\mathbb{R}^n,\mathbb{R}^n)$, let $\mathbf{u}=(-\Delta)^{-1} \mathbf{f}=\Gamma\ast \mathbf{f}$ be the Newtonian potential of $\mathbf{f}$, where $\Gamma$ is defined by
\[
\Gamma(x)=
\begin{cases}
 -\frac{1}{2\pi}\ln|x|,\quad \ \ &n=2,\\
\frac{1}{|\mathbb{S}^{n-1}|(n-2)|x|^{n-2}},\quad \ \ & n\geq 3.
\end{cases}
\]
For $n\geq2$, Bourgain and Brezis \cite{BB2004,BB2007} proved that
\begin{equation}\label{ineq11}
\|\nabla \mathbf{u}\|_{L^{n'}(\mathbb{R}^n)}\leq C(\|\mathbf{f}\|_{L^1(\mathbb{R}^n)}+\|\divg \mathbf{f}\|_{\dot{W}^{-2,n'}(\mathbb{R}^n)}), \ n'=n/(n-1).
\end{equation}
Maz'ya \cite{Mazya2010} established the weighted inequalities related to \eqref{ineq11} as follows:
\begin{enumerate}
\item Let $1\leq q<n'$, $\beta=1-n(1-\frac{1}{q})$ and $\mathbf{f}\in L^1(\mathbb{R}^n,\mathbb{R}^n)$ satisfying $\divg \mathbf{f}=0$. Then it holds that
\begin{equation}\label{ineq13}
\left\|\frac{\nabla \mathbf{u}}{|x|^\beta}\right\|_{L^q(\mathbb{R}^n)}
\leq C\|\mathbf{f}\|_{L^1(\mathbb{R}^n)}.
\end{equation}
\item Let $1<q<n'$, $\beta=1-n(1-\frac{1}{q})$, $\mathbf{f}\in L^1(\mathbb{R}^n,\mathbb{R}^n)$ and $\nabla(-\Delta)^{-1}\divg \mathbf{f}\in L^1(\mathbb{R}^n,\mathbb{R}^n)$. Then it holds that
\begin{equation}\label{ineqM}
\left\|\frac{\nabla \mathbf{u}}{|x|^\beta}\right\|_{L^q(\mathbb{R}^n)}
\leq C(\|\mathbf{f}\|_{L^1(\mathbb{R}^n)}+\left\|\nabla(-\Delta)^{-1}\divg \mathbf{f}\right\|_{L^1(\mathbb{R}^n)}).
\end{equation}
\end{enumerate}
Soon after, Bousquet and Mironescu \cite{BM2011} gave a short proof of Maz'ya's result with improvements. They found that the inequality \eqref{ineqM} holds also for $q=1$. In fact, using Leray decomposition and a similar trick used in the proof of Theorem \ref{ZZB}, we see that the inequality \eqref{ineq13} and the result of Bousquet and Mironescu are equivalent, see Remark \ref{Equivalent}. Inspired by these inequalities, we try to extend the range of exponents that Theorem \ref{SW} holds for weighted vector fields. We introduce a more general fractional integral operator $\mathrm{T}_{\lambda}$ defined by
$$\mathrm{T}_{\lambda}f(x)=K\ast f(x)=\int_{\mathbb{R}^n}K(x-y)f(y)dy,$$
where the kernel $K(x)$ satisfies
\begin{equation}\label{Condition}
(i)\ |K(x)|\leq C|x|^{\lambda-n},\text{ if $|x|\neq0$}; \ (ii)\ |K(x-y)-K(x)|\leq C\frac{|y|}{|x|^{n+1-\lambda}},\text{ if $|y|\leq \frac{|x|}{2}$.}
\end{equation}

Our main result is that
\begin{Thm}\label{ZZB}
Let $n\geq 2$, $0<\lambda<n$, $\alpha<1$, $\beta<\frac{n}{q}$, $\alpha+\beta>0$, $\frac{1}{q}=1+\frac{\alpha+\beta-\lambda}{n}$. Suppose that $K(x)$ satisfies the conditions in \eqref{Condition}. If $1\leq q<\infty$, then
\begin{equation}\label{ineq15}
\left\||x|^{-\beta}\mathrm{T}_{\lambda}\mathbf{f}\right\|_{L^q}\leq C(\left\||x|^{\alpha}\mathbf{f}\right\|_{L^1}+\left\||x|^{\alpha}\nabla(-\Delta)^{-1}\divg \mathbf{f}\right\|_{L^1}).
\end{equation}
\end{Thm}

This paper is organized as follows. In Section 2, we give some notations that have appeared in the context. Section 3 shows the proof of Theorem \ref{ZZB}. In Section 4, applying the new inequality and Stein-Weiss inequality, we give a new proof to Hardy inequality as well as a class of Caffarelli-Kohn-Nirenberg inequalities and obtain new $\divg$-$\curl$ inequalities for vector fields.

\section{Notations and definitions}

Let $\Omega\subseteq\mathbb{R}^n$. The notation $\mathcal{D}(\Omega,\mathbb{R}^n)$ denotes the space of $n$-dimensional vector-valued functions that are infinitely differentiable and have compact supports in $\Omega$. Let $\mathcal{D}'(\Omega,\mathbb{R}^n)$ denote the dual space of $\mathcal{D}(\Omega,\mathbb{R}^n)$. The space $\mathbf{H}^p(\divg,\Omega)$ is defined by
$$\mathbf{H}^p(\divg,\Omega)=\{\mathbf{v}\in L^p(\Omega,\mathbb{R}^n):\divg\mathbf{v}\in L^p(\Omega)\}$$
and is provided with the norm
$$\|\mathbf{v}\|_{\mathbf{H}^p(\divg,\Omega)}=\|\mathbf{v}\|_{L^p(\Omega)}+\|\divg\mathbf{v}\|_{L^p(\Omega)}.$$
We denote by $\mathbf{H}_0^p(\divg,\Omega)$ the closure of $C_c^\infty(\Omega,\mathbb{R}^n)$ in $\mathbf{H}^p(\divg,\Omega)$. It is well-known that $W^{1,p}(\mathbb{R}^n)=W_0^{1,p}(\mathbb{R}^n)$. Proposition \ref{H-div} below implies that $\mathbf{H}^p(\divg,\mathbb{R}^n)$ has the same property, i.e., $\mathbf{H}^p(\divg,\mathbb{R}^n)=\mathbf{H}_0^p(\divg,\mathbb{R}^n)$.

We recall the definition of curl operator. It is defined as a matrix of order $n\geq2$. We denote the elements of the matrix $\CURL \mathbf{v}$ by
 $$\CURL_{ij}\mathbf{v}=\frac{\partial v_i}{\partial x_j}-\frac{\partial v_j}{\partial x_i}, \text{ $i,j=1,2,\cdots,n$, for any $\mathbf{v}=(v_1,v_2,\cdots,v_n)\in \mathcal{D}'(\Omega,\mathbb{R}^n)$.}$$
 For a matrix $\mathbf{A}=(a_{ij})\in \mathcal{D}'(\Omega,\mathbb{R}^{n^2})$, where $i,j=1,2,\cdots,n$, we define its divergence by
 $$\divg \mathbf{A}=(\sum_{j=1}^n\frac{\partial a_{1j}}{\partial x_j},\sum_{j=1}^n\frac{\partial a_{2j}}{\partial x_j},\cdots,\sum_{j=1}^n\frac{\partial a_{nj}}{\partial x_j}).$$
 For any $\mathbf{v}=(v_1,v_2,\cdots,v_n)\in \mathcal{D}'(\Omega,\mathbb{R}^n)$, it holds that
\begin{equation}\label{vectorid}
\divg\CURL \mathbf{v}=\Delta \mathbf{v}-\nabla(\divg \mathbf{v}).
\end{equation}

Throughout this paper, bold typeface will indicate vector or matrix quantities; normal typeface will be used for vector and matrix components and for scalars. To simplify the notations, we write $\|\cdot\|_{L^p}$ instead of $\|\cdot\|_{L^p(\mathbb{R}^n)}$.

\section{Proofs of Theorem \ref{ZZB}}

\begin{Prop}\label{H-div}
Let $1\leq p<\infty$. Then $\mathbf{H}^p(\divg,\mathbb{R}^n)=\mathbf{H}_0^p(\divg,\mathbb{R}^n)$.
\end{Prop}
\begin{proof}
It suffices to show that $C_c^\infty(\mathbb{R}^n,\mathbb{R}^n)$ is dense in $\mathbf{H}^p(\divg,\mathbb{R}^n)$. Assume that $\mathbf{v}\in \mathbf{H}^p(\divg,\mathbb{R}^n)$. Set $\mathbf{v}^\varepsilon=\eta_\varepsilon\ast\mathbf{v}$, where $\eta$ is the standard mollifier. We have that $\mathbf{v}^\varepsilon\in C^\infty(\mathbb{R}^n,\mathbb{R}^n)$ and $\mathbf{v}^\varepsilon\rightarrow \mathbf{v}$ in $\mathbf{H}^p(\divg,\mathbb{R}^n)$ as $\varepsilon\rightarrow 0$. Let $\zeta\in C_c^\infty(B_2(0))$ be a cut-off function such that $0\leq\zeta\leq 1$ and $\zeta\equiv 1$ in $B_1(0)$. Set $\mathbf{v}^\varepsilon_k(x)=\mathbf{v}^\varepsilon(x)\zeta(\frac{x}{k})$, then $\mathbf{v}^\varepsilon_k\in C_c^\infty(\mathbb{R}^n,\mathbb{R}^n)$. By the definition of divergence operator, we get
$$\divg\mathbf{v}^\varepsilon_k(x)=\divg\mathbf{v}^\varepsilon(x)\zeta(\frac{x}{k})+\frac{1}{k}\mathbf{v}^\varepsilon(x)\cdot(\nabla\zeta)(\frac{x}{k}).$$ Therefore, for any fixed $\varepsilon$, we have
$$\left\|\divg\mathbf{v}^\varepsilon_k-\divg\mathbf{v}^\varepsilon\right\|_{L^p}^p\leq C\left(\int_{|x|>k}|\divg\mathbf{v}^\varepsilon|^pdx+\frac{1}{k}\|\mathbf{v}^\varepsilon\|_{L^p}^p\right)\rightarrow 0\text{ as $k\rightarrow \infty$}$$
and
$$\left\|\mathbf{v}^\varepsilon_k-\mathbf{v}^\varepsilon\right\|_{L^p}^p\leq\int_{|x|>k}|\mathbf{v}^\varepsilon|^pdx\rightarrow 0\text{ as $k\rightarrow \infty$.}$$
By the above two inequalities and using the fact that $\mathbf{v}^\varepsilon\rightarrow \mathbf{v}$ in $\mathbf{H}^p(\divg,\mathbb{R}^n)$ as $\varepsilon\rightarrow 0$, we know that for any $\varepsilon>0$, there exists $k=k(\varepsilon)$ such that $\mathbf{v}^\varepsilon_k\rightarrow \mathbf{v}$ in $\mathbf{H}^p(\divg,\mathbb{R}^n)$ as $\varepsilon\rightarrow 0$.
\end{proof}

Since $C_c^\infty(\mathbb{R}^n,\mathbb{R}^n)$ is dense in $\mathbf{H}^1(\divg,\mathbb{R}^n)$, by the divergence theorem and an approximation argument, we obtain the following corollary.
\begin{Cor}\label{Gauss0}
For any $\mathbf{v}\in\mathbf{H}^1(\divg,\mathbb{R}^n)$, it holds that
$$\int_{\mathbb{R}^n}\divg \mathbf{v}dx=0.$$
\end{Cor}

Let $\rho_0\in C_{c}^{\infty}(\mathbb{R}^{+})$ be a cut-off function such that $0\leq\rho_0\leq 1$ and
\[
\rho_0(t)=
\begin{cases}
 1,\quad \ \ &t\leq\frac{1}{4},\\
0,\quad \ \ & t\geq\frac{1}{2}.
\end{cases}
\]
We denote $\rho(y,x)=\rho_0(\frac{|y|}{|x|})$ for $(y,x)\in\mathbb{R}^n\times(\mathbb{R}^n\backslash\{0\})$. We extract a lemma from the proof of the main theorem in \cite{BM2011}, but in a simple case.
\begin{Lem}\label{BMlemma}
Let $n\geq 2$, $\mathbf{f}\in L^1_{loc}(\mathbb{R}^n,\mathbb{R}^n)$ with $\divg\mathbf{f}=0$. Then we have
$$\left|\int_{\mathbb{R}^n}\rho(y, x)\mathbf{f}(y)dy\right|\leq C\int_{|y|\leq\frac{|x|}{2}}\frac{|y|}{|x|}|\mathbf{f}(y)|dy.$$
\end{Lem}
\begin{proof}
For any $|x|\neq 0$, we have $y_i\rho(y, x)\mathbf{f}(y)\in\mathbf{H}^1(\divg,\mathbb{R}^n)$ and
$$\divg(y_i\rho(y, x)\mathbf{f}(y))=\nabla_y(y_i\rho(y, x))\cdot\mathbf{f}(y)+y_i\rho(y, x)\divg\mathbf{f}=\nabla_y(y_i\rho(y, x))\cdot\mathbf{f}(y).$$
By Corollary \ref{Gauss0}, we get
$$\int_{\mathbb{R}^n}\divg(y_i\rho(y, x)\mathbf{f}(y))dy=0.$$
Thus
$$
\int_{\mathbb{R}^n}\rho(y, x)f_i(y)+\frac{y_i}{|y||x|}\rho_0'(\frac{|y|}{|x|})\sum_{j=1}^n y_jf_j(y)dy=0.
$$
So we get
$$\left|\int_{\mathbb{R}^n}\rho(y, x)f_i(y)dy\right|\leq C\int_{|y|\leq\frac{|x|}{2}}\frac{|y|}{|x|}|\mathbf{f}(y)|dy.$$
\end{proof}

\begin{proof}[Proof of Theorem \ref{ZZB}]
First, we claim that Theorem \ref{ZZB} is equivalent to the following statement.
\begin{Claim}\label{claim}
Let $n\geq 2$, $0<\lambda<n$, $\alpha<1$, $\beta<\frac{n}{q}$, $\alpha+\beta>0$, $\frac{1}{q}=1+\frac{\alpha+\beta-\lambda}{n}$ and $\divg \mathbf{f}=0$. If $1\leq q<\infty$, then
\begin{equation}\label{Zineq1}
\left\||x|^{-\beta}\mathrm{T}_{\lambda}\mathbf{f}\right\|_{L^q}\leq C\left\||x|^{\alpha}\mathbf{f}\right\|_{L^1}.
\end{equation}
\end{Claim}
It is easy to see that Claim \ref{claim} is a special case of Theorem \ref{ZZB}. We only
need to derive Theorem \ref{ZZB} from  Claim \ref{claim}. We decompose $\mathbf{f}$ as $\mathbf{f}=\mathbf{f}+\nabla(-\Delta)^{-1}\divg \mathbf{f}-\nabla(-\Delta)^{-1}\divg \mathbf{f}$, which is called the Leray decomposition of $\mathbf{f}$. $\mathbf{f}+\nabla(-\Delta)^{-1}\divg \mathbf{f}$ is the divergence-free part while $\nabla(-\Delta)^{-1}\divg \mathbf{f}$ is the curl-free part. Since $\divg(\mathbf{f}+\nabla(-\Delta)^{-1}\divg \mathbf{f})=0$, by Claim \ref{claim}, we obtain
\begin{equation}\label{ineq31}
\aligned
\left\||x|^{-\beta}\mathrm{T}_{\lambda}(\mathbf{f}+\nabla(-\Delta)^{-1}\divg \mathbf{f})\right\|_{L^q}&\leq C\left\||x|^{\alpha}(\mathbf{f}+\nabla(-\Delta)^{-1}\divg \mathbf{f})\right\|_{L^1}\\
&\leq C(\||x|^{\alpha}\mathbf{f}\|_{L^1}+\left\||x|^{\alpha}\nabla(-\Delta)^{-1}\divg \mathbf{f}\right\|_{L^1}).
\endaligned
\end{equation}
On the other hand, for $1\leq i<j\leq n$, set $(g_i,g_j)=(\frac{\partial}{\partial x_j}(-\Delta)^{-1}\divg \mathbf{f},-\frac{\partial}{\partial x_i}(-\Delta)^{-1}\divg \mathbf{f}$), $g_k=0$ if $k\neq i,j$. Then we have $\divg \mathbf{g}=0$.
By Claim \ref{claim}, we obtain
$$
\aligned
&\left\||x|^{-\beta}\mathrm{T}_{\lambda}(\frac{\partial}{\partial x_i}(-\Delta)^{-1}\divg \mathbf{f})\right\|_{L^q}+\left\||x|^{-\beta}\mathrm{T}_{\lambda}(\frac{\partial}{\partial x_j}(-\Delta)^{-1}\divg \mathbf{f})\right\|_{L^q}\\
&\leq 2\left\||x|^{-\beta}\mathrm{T}_{\lambda}\mathbf{g}\right\|_{L^q}\leq C\||x|^{\alpha}\mathbf{g}\|_{L^1}\leq C\left\||x|^{\alpha}\nabla(-\Delta)^{-1}\divg \mathbf{f}\right\|_{L^1}.
\endaligned
$$
Hence, we have
\begin{equation}\label{ineq32}
\left\||x|^{-\beta}\mathrm{T}_{\lambda}(\nabla(-\Delta)^{-1}\divg \mathbf{f})\right\|_{L^q}\leq C\left\||x|^{\alpha}\nabla(-\Delta)^{-1}\divg \mathbf{f}\right\|_{L^1}.
\end{equation}
Therefore, the inequality \eqref{Zineq1} follows from the inequalities \eqref{ineq31} and \eqref{ineq32}.

Hence, we only need to prove the case of $\divg \mathbf{f}=0$. The benefit of this observation is that there is no need to deal with the term $\int_{\mathbb{R}^n}y_i\rho(y, x)\divg\mathbf{f}(y)dy$ in Lemma \ref{BMlemma} for $\mathbf{f}$ is a divergence-free vector field, which is different from \cite{BM2011}.

The proof of Claim \ref{claim} consists of the following steps.

Step 1. We write $\mathrm{T}_{\lambda}\mathbf{f}(x)=J_1(x)+J_2(x)$, where
$$
J_1(x)=\int_{\mathbb{R}^n}\rho(y, x)K(x-y)\mathbf{f}(y)dy,\ J_2(x)=\int_{\mathbb{R}^n}(1-\rho(y, x))K(x-y)\mathbf{f}(y)dy.
$$
By the condition (i) in \eqref{Condition} and generalized Minkowski's inequality, we have
$$
\aligned
 \left\||x|^{-\beta}J_2(x)\right\|_{L^q}&\leq C\left\||x|^{-\beta}\int_{|y|\geq\frac{|x|}{4}}\frac{|\mathbf{f}(y)|}{|x-y|^{n-\lambda}}dy\right\|_{L^q}\\
 &\leq C\int_{\mathbb{R}^n}\left(\int_{|x|\leq 4|y|}\frac{|x|^{-\beta q}}{|x-y|^{(n-\lambda)q}}dx\right)^{\frac{1}{q}}|\mathbf{f}(y)|dy.
\endaligned
$$
Since
$$
\aligned
\int_{|x|\leq 4|y|}\frac{|x|^{-\beta q}}{|x-y|^{(n-\lambda)q}}dx&=\int_{|x|\leq \frac{|y|}{2}}\frac{|x|^{-\beta q}}{|x-y|^{(n-\lambda)q}}dx+\int_{\frac{|y|}{2}\leq|x|\leq 4|y|}\frac{|x|^{-\beta q}}{|x-y|^{(n-\lambda)q}}dx\\
&\leq C\left(|y|^{(\lambda-n)q}\int_{|x|\leq \frac{|y|}{2}}|x|^{-\beta q}dx+|y|^{-\beta q}\int_{|x-y|\leq 5|y|}|x-y|^{(\lambda-n)q}dx\right)\\
&=C|y|^{n-\beta q+(\lambda-n)q}=C|y|^{\alpha q},
\endaligned
$$
here we require $n-\beta q>0$ and $n+(\lambda-n)q>0$, then we get
\begin{equation}\label{ineq35}
\left\||x|^{-\beta}J_2(x)\right\|_{L^q}\leq C\int_{\mathbb{R}^n}|y|^\alpha|\mathbf{f}(y)|dy.
\end{equation}

Step 2. We write $J_1(x)=J_{11}(x)+J_{12}(x)$, where
$$J_{11}(x)=\int_{\mathbb{R}^n}\rho(y, x)(K(x-y)-K(x))\mathbf{f}(y)dy,\ J_{12}(x)=\int_{\mathbb{R}^n}\rho(y, x)K(x)\mathbf{f}(y)dy.$$
Thus by generalized Minkowski's inequality and the condition (ii) in \eqref{Condition}, we obtain
\begin{equation}\label{ineq36}
\aligned
\left\||x|^{-\beta}J_{11}(x)\right\|_{L^q}&\leq \int_{\mathbb{R}^n}\left(\int_{\mathbb{R}^n}(\rho(y, x)|x|^{-\beta}|K(x-y)-K(x)|)^q dx\right)^{\frac{1}{q}}|\mathbf{f}(y)|dy\\
&\leq C\int_{\mathbb{R}^n}\left(\int_{|x|\geq 2|y|}|x|^{(\lambda-n-1-\beta) q}dx\right)^{\frac{1}{q}}|y||\mathbf{f}(y)|dy\\
&=C\int_{\mathbb{R}^n}(|y|^{(\alpha-1)q})^{\frac{1}{q}}|y||\mathbf{f}(y)|dy=C\int_{\mathbb{R}^n}|y|^\alpha|\mathbf{f}(y)|dy,
\endaligned
\end{equation}
here we require $n+(\lambda-n-1-\beta)q<0$, i.e., $\alpha<1$.

Step 3. At last we deal with the term $J_{12}(x)$. By the condition (i) in \eqref{Condition}, we have
$$|J_{12}(x)|\leq C\frac{1}{|x|^{n-\lambda}}\left|\int_{\mathbb{R}^n}\rho(y, x)\mathbf{f}(y)dy\right|.$$
Using Lemma \ref{BMlemma} and generalized Minkowski's inequality, we get
\begin{equation}\label{ineq37}
\aligned
 \left\||x|^{-\beta}J_{12}(x)\right\|_{L^q}&\leq C\left\||x|^{\lambda-n-\beta}\int_{|y|\leq\frac{|x|}{2}}\frac{|y|}{|x|}|\mathbf{f}(y)|dy\right\|_{L^q}\\
 &\leq C\int_{\mathbb{R}^n}\left(\int_{|x|\geq 2|y|}|x|^{(\lambda-n-\beta-1)q}dx\right)^{\frac{1}{q}}|y||\mathbf{f}(y)|dy\\
 &=C\int_{\mathbb{R}^n}|y|^\alpha|\mathbf{f}(y)|dy,
\endaligned
\end{equation}
here we require $n+(\lambda-n-\beta-1)q<0$.

Combining the inequalities \eqref{ineq35}, \eqref{ineq36} and \eqref{ineq37}, we obtain the result.
\end{proof}

\begin{Rem}
Let $\mathbf{u}=(-\Delta)^{-1} \mathbf{f}$. For $n\geq3$, using Theorem \ref{ZZB} with $K(x)=|x|^{2-n}$, we have
$$
\left\|\frac{\mathbf{u}}{|x|^\beta}\right\|_{L^q}
\leq C(\|\mathbf{f}\|_{L^1}+\|\nabla(-\Delta)^{-1}\divg \mathbf{f}\|_{L^1}),
$$
where $1\leq q<\frac{n}{n-2}$, $\beta=2+n(\frac{1}{q}-1)$. If we set $K(x)=\frac{x_j}{|x|^n}$ in Theorem \ref{ZZB}, then we can see that our result generalizes the inequality \eqref{ineqM} in a doubly weighted form.
\end{Rem}

\begin{Rem}\label{Equivalent}
Applying inequality \eqref{ineq13} to $\mathbf{g}$ as the proof of Theorem \ref{ZZB}, we can get
$$\left\|\frac{\nabla(-\Delta)^{-1} \nabla(-\Delta)^{-1}\divg \mathbf{f}}{|x|^\beta}\right\|_{L^q}
\leq C\left\|\nabla(-\Delta)^{-1}\divg \mathbf{f}\right\|_{L^1},$$
where $1\leq q<n'$, $\beta=1-n(1-\frac{1}{q})$. Using this inequality, we can derive inequality \eqref{ineqM} for $1\leq q<n'$ from inequality \eqref{ineq13} by the same method used in the proof of Theorem \ref{ZZB}.
\end{Rem}

\section{Applications}
There are many proofs to Hardy inequality (cf. \cite[p. 111]{DDE2012}). Here we give a new proof of Hardy inequality by the theory of singular integrals. If $1<p<n$, we can use Theorem \ref{SW} to prove Hardy inequality
$$\left\|\frac{u}{|x|}\right\|_{L^p}\leq C\|\nabla u\|_{L^p} \text{  for any  $u\in C_{c}^{\infty}(\mathbb{R}^n)$.}$$
In fact, for any $u\in C_{c}^{\infty}(\mathbb{R}^n)$, we have the equality (see \cite[Lemma 7.14]{GT1998})
$$u(x)=\frac{1}{|\mathbb{S}^{n-1}|}\int_{\mathbb{R}^n}\frac{(x-y)\cdot\nabla u(y)}{|x-y|^n}dy.$$
If $p>1$, by Theorem \ref{SW}, we have
$$\left\|\frac{u}{|x|}\right\|_{L^p}\leq \frac{1}{|\mathbb{S}^{n-1}|}\left\|\frac{1}{|x|}\int_{\mathbb{R}^n}\frac{|\nabla u(y)|}{|x-y|^{n-1}}dy\right\|_{L^p}\leq C\|\nabla u\|_{L^p}.$$
But Theorem \ref{SW} doesn't work when $p=1$. Our theorem make it possible to prove the case $p=1$. In view of \eqref{vectorid}, for any function $\mathbf{f}\in C_c^\infty(\mathbb{R}^n,\mathbb{R}^n)$, we have $$\left\||x|^{\alpha}\mathbf{f}\right\|_{L^1}+\left\||x|^{\alpha}\nabla(-\Delta)^{-1}\divg\mathbf{f}\right\|_{L^1}\leq 2(\left\||x|^{\alpha}\mathbf{f}\right\|_{L^1}+\left\||x|^{\alpha}\divg(-\Delta)^{-1}\CURL\mathbf{f}\right\|_{L^1}).$$
Therefore, for divergence-free or curl-free smooth vector fields with compact support, the term $\left\|\nabla(-\Delta)^{-1}\divg\mathbf{f}\right\|_{L^1}$ can be removed in \eqref{ineq15}.
If $p=1$, since
$$\nabla(-\Delta)^{-1}\divg\nabla u=-\nabla u \text{ or }  \CURL\nabla u=\mathbf{0},$$
setting $K(x)=\frac{x_j}{|x|^n}$ in Theorem \ref{ZZB}, we get
$$
\aligned
\left\|\frac{u}{|x|}\right\|_{L^1}&=\frac{1}{|\mathbb{S}^{n-1}|}\left\|\frac{1}{|x|}\int_{\mathbb{R}^n}\frac{(x-y)\cdot\nabla u(y)}{|x-y|^n}dy\right\|_{L^1}\\
&\leq \frac{1}{|\mathbb{S}^{n-1}|}\sum_{j=1}^n\left\|\frac{1}{|x|}\int_{\mathbb{R}^n}\frac{x_j-y_j}{|x-y|^n}\nabla u(y)dy\right\|_{L^1}\\
&\leq C\|\nabla u\|_{L^1}.
\endaligned
$$
Moreover, by the same idea, we can give a new proof of a class of Caffarelli-Kohn-Nirenberg inequalities (see \cite{CKN1984}) as follows:
\begin{enumerate}
\item Let $n\geq 2$. If $\alpha<1$, $\beta<\frac{n}{q}$, $0<\alpha+\beta\leq 1$ and $\frac{1}{q}=1+\frac{\alpha+\beta-1}{n}$, then
\begin{equation}\label{CKN1}
\left\|\frac{u}{|x|^{\beta}}\right\|_{L^q}\leq C\left\||x|^{\alpha}\nabla u\right\|_{L^1}.
\end{equation}
\item If $1<p<\infty$, $\alpha<\frac{n}{p'}$, $\beta<\frac{n}{q}$, $\alpha+\beta\geq 0$ and $\frac{1}{q}=\frac{1}{p}+\frac{\alpha+\beta-1}{n}$, $p\leq q<\infty$, then
\begin{equation}\label{CKNp}
\left\|\frac{u}{|x|^{\beta}}\right\|_{L^q}\leq C\left\||x|^{\alpha}\nabla u\right\|_{L^p}.
\end{equation}
\end{enumerate}
Take $\alpha=0$ in the inequality \eqref{CKN1}, we get
$$\left\|\frac{u}{|x|^{\beta}}\right\|_{L^q}\leq C\left\|\nabla u\right\|_{L^1},$$
where $n\geq2$, $\beta=1-n(1-\frac{1}{q})$ and $1\leq q<n'$.

Another application is to establish some new weighted $\divg$-$\curl$ inequalities. By the way, we point out that $\divg$-$\curl$ inequalities involving $L^1$ norm have been studied by \cite{BB2004}, \cite{LS2005}, \cite{MM2009}, \cite{VanSch2010}, \cite{Lou2014}, \cite{XZ2015}, and the references therein.

\begin{Thm}\label{ZZBn}
Let $\mathbf{u}\in C^\infty_c(\mathbb{R}^n,\mathbb{R}^n)$.
\begin{enumerate}
\item Let $n\geq 3$, $\alpha<1$, $\beta<\frac{n}{q}$, $0<\alpha+\beta\leq 1$ and $\frac{1}{q}=1+\frac{\alpha+\beta-1}{n}$. If $\divg \mathbf{u}=0$, then
    $$\left\|\frac{\mathbf{u}}{|x|^{\beta}}\right\|_{L^q}\leq C\left\||x|^{\alpha}\CURL \mathbf{u}\right\|_{L^1}.$$
\item Let $n\geq 2$, $1<p<\infty$, $\alpha<\frac{n}{p'}$, $\beta<\frac{n}{q}$, $\alpha+\beta\geq 0$ and $\frac{1}{q}=\frac{1}{p}+\frac{\alpha+\beta-1}{n}$. If $p\leq q<\infty$, then
$$\left\|\frac{\mathbf{u}}{|x|^{\beta}}\right\|_{L^q}\leq C(\left\||x|^{\alpha}\divg \mathbf{u}\right\|_{L^p}+\left\||x|^{\alpha}\CURL \mathbf{u}\right\|_{L^p}).$$
\end{enumerate}
\end{Thm}
\begin{proof}
By Green's representation formula and the identity \eqref{vectorid}, we obtain
\begin{equation}\label{vector-id2}
\aligned
\mathbf{u}(x)&=\int_{\mathbb{R}^n}\Gamma(x-y)(-\Delta \mathbf{u})(y)dy\\
&=\int_{\mathbb{R}^n}\Gamma(x-y)(-\divg\CURL\mathbf{u}-\nabla(\divg \mathbf{u}))(y)dy\\
&=\int_{\mathbb{R}^n}\nabla_y\Gamma(x-y)\cdot\CURL \mathbf{u}(y)+\nabla_y\Gamma(x-y)\divg \mathbf{u}(y)dy\\
&=\frac{1}{|\mathbb{S}^{n-1}|}\int_{\mathbb{R}^n}\frac{x-y}{|x-y|^n}\cdot\CURL \mathbf{u}(y)+\frac{x-y}{|x-y|^n}\divg \mathbf{u}(y)dy,
\endaligned
\end{equation}
where the dot product between a vector $\mathbf{v}$ and a matrix $\mathbf{A}=(\mathbf{A}_1,\mathbf{A}_2,\cdots,\mathbf{A}_n)^T$ is defined by $\mathbf{v}\cdot \mathbf{A}=(\mathbf{v}\cdot \mathbf{A}_1,\mathbf{v}\cdot \mathbf{A}_2,\cdots,\mathbf{v}\cdot \mathbf{A}_n)$.

If $p>1$, then by Theorem \ref{SW}, we obtain the second inequality.

Next we turn to the case $p=1$ and $\divg \mathbf{u}=0$. For $1\leq i<j<k\leq n$, set $(f_i,f_j,f_k)=(\frac{\partial}{\partial y_i},\frac{\partial}{\partial y_j},\frac{\partial}{\partial y_k})\times (u_i,u_j,u_k)$, $f_l=0$ if $l\neq i,j,k$. Then we have $\divg \mathbf{f}=0$.
Applying Theorem \ref{ZZB} to $\mathbf{f}$, then we obtain
$$
\left\|\frac{1}{|x|^\beta}\int_{\mathbb{R}^n}\frac{x_m-y_m}{|x-y|^n}(\frac{\partial u_i}{\partial y_j}-\frac{\partial u_j}{\partial y_i})(y)dy\right\|_{L^q}\leq C\left\||x|^{\alpha}\mathbf{f}\right\|_{L^1}\leq C\left\||x|^{\alpha}\CURL \mathbf{u}\right\|_{L^1},
$$
for any $1\leq m\leq n$. Then using \eqref{vector-id2} with $\divg \mathbf{u}=0$, we get the first inequality.
\end{proof}

We can give another proof to the second inequality in Theorem \ref{ZZBn}. We first state the following two facts:
\begin{enumerate}
\item For any $\mathbf{u}\in C^\infty_c(\mathbb{R}^n,\mathbb{R}^n)$, it holds that
$$\nabla\mathbf{u}=-\nabla(-\Delta)^{-1}\divg\CURL \mathbf{u}-\nabla(-\Delta)^{-1}\nabla(\divg \mathbf{u}),$$
\item If $1<p<\infty$ and $-\frac{n}{p}<\alpha<\frac{n}{p'}$, then $|x|^{\alpha p}$ is in the class $A_p$ (see \cite[Proposition 1.4.4]{LDY2007}).
\end{enumerate}
Using the above two results and the idea of proof given in \cite{JWX2013}, we can generalize \cite[Lemma 2.4]{JWX2013} to any dimension $n\geq 2$:

If $1<p<\infty$ and $-\frac{n}{p}<\alpha<\frac{n}{p'}$, then for any $\mathbf{u}\in C^\infty_c(\mathbb{R}^n,\mathbb{R}^n)$, we have the following weighted inequality for div-curl-grad operators
\begin{equation}
\left\||x|^{\alpha}\nabla \mathbf{u}\right\|_{L^p}\leq C(\left\||x|^{\alpha}\divg \mathbf{u}\right\|_{L^p}+\left\||x|^{\alpha}\CURL \mathbf{u}\right\|_{L^p}).
\end{equation}
By Caffarelli-Kohn-Nirenberg inequality \eqref{CKNp} and the above inequality, we can also obtain the second conclusion of Theorem \ref{ZZBn}. However, this method is not applicable for the case $p=1$.

\subsection*{Acknowledgements.}
First, I would like to express my gratitude to my supervisor Prof. Xingbin Pan for guidance and constant encouragement. I also would like to thank Dr. Xingfei Xiang for introducing me some inequalities involving $L^1$-norm, Deliang Chen, Yong Zeng and Dr. Huyuan Chen for useful discussions and suggestions. The work was partly supported by the National Natural Science Foundation of China grant no. 11171111 and by Outstanding Doctoral Dissertation Cultivation Plan of Action (PY2015038).


\vspace {0.5cm}
\begin {thebibliography}{DUMA}


\bibitem{BB2002} J. Bourgain, H. Brezis, {\it  On the equation $\divg Y = f$ and application to control
of phases}, J. Amer. Math. Soc. {\bf 16} (2), (2002) 393-426.

\bibitem{BB2004} J. Bourgain, H. Brezis, {\it  New estimates for the Laplacian, the div-curl, and
related Hodge systems}, C. R. Math. Acad. Sci. Paris {\bf 338}, (2004) 539-543.

\bibitem{BB2007} J. Bourgain, H. Brezis, {\it  New estimates for elliptic
equations and Hodge type systems}, J. Eur. Math. Soc.  {\bf 9} (2), (2007)
277-315.

\bibitem{BM2011} P. Bousquet, P. Mironescu,
{\it An elementary proof of an inequality of Maz'ya involving $L^1$ vector fields}, Nonlinear elliptic partial differential equations, 59-63, Contemp. Math. {\bf 540}, Amer. Math. Soc., Providence, RI, 2011.

\bibitem{CKN1984} L. A. Caffarelli, R. Kohn, L. Nirenberg, {\it  First order interpolation inequalities with weights}, Compos. Math. {\bf 53} (3), (1984) 259-275.

\bibitem{DDE2012} F. Demengel, G. Demengel, {\it Functional spaces for the theory of elliptic partial differential equations}, Universitext, Springer, London, 2012, Translated from the 2007 French original by Reinie Ern\'{e}.

\bibitem{DDD2011} P. L. De N\'{a}poli, I. Drelichman, R. G. Dur\'{a}n, {\it On weighted inequalities for fractional integrals of radial functions}, Ill. J. Math. {\bf 55} (2), (2011) 575-587.

\bibitem{GT1998} D. Gilbarg, N. S. Trudinger, {\it  Elliptic Partial Differential Equations of Second
Order}, Classics in Mathematics, Springer-Verlag, Berlin, 2001, Reprint of the 1998 edition.

\bibitem{JWX2013} Q. S. Jiu, Y. Wang, Z. P. Xin, {\it Global well-posedness of the Cauchy problem of two-dimensional compressible Navier-Stokes equations in weighted spaces}, J. Differential Equations {\bf 255} (3), (2013) 351-404.

\bibitem{LS2005} L. Lanzani, E. M. Stein, {\it A note on div curl inequalities}, Math. Res. Lett. {\bf 12} (1), (2005) 57-61.

\bibitem{Lou2014} A. Loulit, {\it Weighted estimates for $L^1$-vector fields}, Proc. Amer. Math. Soc. {\bf 142} (12), (2014) 4171-4179.

\bibitem{LDY2007} S. Z. Lu, Y. Ding, D. Y. Yan, {\it Singular integrals and related topics}, World Scientific, Singapore, 2007.

\bibitem{Mazya2010} V. Maz'ya, {\it  Estimates for differential operators of vector analysis involving $L^1$-norm}, J. Eur. Math. Soc. {\bf 12} (1), (2010) 221-240.

\bibitem{MM2009} I. Mitrea, M. Mitrea, {\it A remark on the regularity of the div-curl system}, Proc. Amer. Math. Soc. {\bf 137} (5), (2009) 1729-1733.

\bibitem{SteinW} E. M. Stein, G. Weiss, {\it Fractional integrals on n-dimensional Euclidean space,} J. Math. Mech. {\bf 7}, (1958) 503-514.

\bibitem{VanSch2010} J. Van Schaftingen, {\it Limiting fractional and Lorentz space estimates of differential forms}, Proc. Amer. Math. Soc. {\bf 138} (1), (2010) 235-240.

\bibitem{XZ2015} X. F. Xiang, Z. B. Zhang, {\it Hardy-type inequalities for vector fields with vanishing tangential components}, Proc. Amer. Math. Soc. {\bf 143} (12), (2015) 5369-5379.

\end{thebibliography}
\end{document}